\documentclass[12pt]{article}
\usepackage{mathrsfs}
\usepackage{amsfonts}

\usepackage{amsmath}
\usepackage{amssymb}
\usepackage{amsthm}
\usepackage{fancyhdr}

\usepackage{todonotes}

\usepackage{geometry}
\theoremstyle{plain}
\newtheorem{theorem}{Theorem}[section]

\newtheorem{lemma}[theorem]{Lemma}

\theoremstyle{definition}

\newtheorem{example}[theorem]{Example}

\def\UU{{\mathcal U}}
\def\FF{{\mathcal F}}

\usepackage{latexsym}
\usepackage{amssymb}
\usepackage{fancyvrb}
\usepackage{color}
\usepackage{hyperref}
\usepackage[mathscr]{eucal}

\usepackage{indentfirst}

\begin{document}

%\linenumbers

\begin{center}
{\Large \bf     On the partial $ \Pi  $-property of second minimal or second maximal subgroups of Sylow subgroups of finite groups

\renewcommand{\thefootnote}{\fnsymbol{footnote}}

\footnotetext[1]
{Corresponding author.}

}\end{center}

                        \vskip0.6cm
\begin{center}

                       Zhengtian Qiu, Jianjun Liu and Guiyun Chen$^{\ast}$

                            \vskip0.5cm

       School of Mathematics and Statistics, Southwest University

       Chongqing 400715, P. R. China

E-mail addresses:  qztqzt506@163.com \, \  liujj198123@163.com   \, \ gychen1963@163.com

\end{center}

                          \vskip0.5cm

\begin{abstract}
	Let $ H $ be a subgroup of a finite group $ G $.  We say that $ H $ satisfies the partial $ \Pi  $-property in $ G $ if if there exists a chief series $ \varGamma_{G}: 1 =G_{0} < G_{1} < \cdot\cdot\cdot < G_{n}= G $ of $ G $ such that for every $ G $-chief factor $ G_{i}/G_{i-1}$ $ (1\leq i\leq n) $ of $ \varGamma_{G} $, $ | G / G_{i-1} : N _{G/G_{i-1}} (HG_{i-1}/G_{i-1}\cap G_{i}/G_{i-1})| $ is a $ \pi (HG_{i-1}/G_{i-1}\cap G_{i}/G_{i-1}) $-number.  In this paper, we study the influence of some   second minimal or second maximal subgroups of a Sylow subgroup satisfying  the partial $ \Pi  $-property  on the structure of a finite group.\\
	
	\small \textbf{Keywords:} Finite group, $ p $-soluble group, the partial  $ \Pi $-property, $ p $-length.\\
	
	\small \textbf{Mathematics Subject Classification (2020):} 20D10, 20D20.
\end{abstract}

%{\hspace{0.88cm} \small \textbf{Keywords:} Finite group; $ p $-soluble group, the partial  $ \Pi $-property, $ p $-length.}
	
%\vskip0.1in

%{\hspace{0.88cm} \small \textbf{Mathematics Subject Classification (2020):} 20D10;   20D20.}

%\vskip1cm
\section{Introduction}

    All groups considered in this paper are finite. We use conventional notions as in \cite{Huppert-1967}. $ G $ always denotes a finite group, $ p $ denotes a fixed prime, $ \pi $ denotes some set of primes  and $ \pi(G) $ denotes the set of all primes dividing $ |G| $. An integer $ n $ is called a $ \pi $-number if all prime divisors of $ n $ belong to $ \pi $. In particular, an integer $ n $ is called a $ p $-number if it is a power of $ p $.

    Recall that a class $ \FF $  of groups is called a formation if $ \FF $ is closed under taking homomorphic images and subdirect products.  A formation $ \FF $ is said to be saturated if $ G/\Phi(G) \in \FF $  implies that $ G \in \FF $.  Throughout  this paper, we  use $ \UU $ (resp. $ \UU_{p} $) to denote the class of supersoluble (resp. $ p $-supersoluble) groups.

    Let $ \FF $ be a formation. The $ \FF $-residual of $ G $, denoted by $ G^{\FF} $, is the smallest normal subgroup of $ G $
    with quotient in $ \FF $. A chief factor $ L/K $ of $ G $ is said to be $ \FF $-central  (resp. $ \FF $-eccentric) in $ G $ if $ L/K \rtimes G/C_{G}(L/K) \in \FF $ (resp. $ L/K \rtimes G/C_{G}(L/K) \not \in \FF $). In particular, a chief factor $ L/K $ of  $ G $ is $ \UU $-central in $ G $  if and only if $ L/K $ is cyclic.  A normal subgroup $ N $ of $ G $ is called $ \FF $-hypercentral in $ G $ if either $ N = 1 $ or every $ G $-chief factor below $ N $ is
    $ \FF $-central in $ G $. Let $ Z_{\FF}(G) $ denotes the $ \FF $-hypercenter of $ G $, that is, the product of all $ \FF $-hypercentral normal subgroups of $ G $. Recall that  a group $ H $ is said to be quasisimple if $ H/Z(H) $ is simple and $ H $ is perfect. A subnormal quasisimple subgroup of $ G $ is
    called a component of $ G $.  Let $ U $ and $ V $ be distinct components of $ G $. Then $ [U,V] = 1 $.  The components of $ G $ normalize each other, and so each component is normal in the subgroup generated by all of them. This subgroup, denoted $ E(G) $, is called the layer of $ G $. The product $ F(G)E(G) $ is the generalised Fitting subgroup of $ G $. It is usually denoted by $ F^{*}(G) $ (for details, see \cite[Chapter 9]{isaacs2008finite}).

    In \cite{Chen-2013}, Chen and Guo introduced the concept of the partial  $ \Pi $-property of subgroups of finite groups,  which generalizes a large number of known embedding
    properties (see \cite[Section 7]{Chen-2013}). Let $ H $ be a subgroup of  $ G $. We say that $ H $ satisfies the partial $ \Pi  $-property in $ G $ if if there exists a chief series $ \varGamma_{G}: 1 =G_{0} < G_{1} < \cdot\cdot\cdot < G_{n}= G $ of $ G $ such that for every $ G $-chief factor $ G_{i}/G_{i-1} $ $ (1\leq i\leq n) $ of $ \varGamma_{G} $, $ | G / G_{i-1} : N _{G/G_{i-1}} (HG_{i-1}/G_{i-1}\cap G_{i}/G_{i-1})| $ is a $ \pi (HG_{i-1}/G_{i-1}\cap G_{i}/G_{i-1}) $-number. They proved the following results by assuming  some maximal or minimal subgroups of  a Sylow subgroup  satisfiy the partial $ \Pi $-property.

    \begin{theorem}[{\cite[Proposition 1.4]{Chen-2013}}]\label{maximal}
     Let $ E $ be a normal subgroup of $ G $ and let $ P $ be a Sylow $ p $-subgroup of $ E $. If every maximal subgroup of $ P $ satisfies the partial $ \Pi $-property in $ G $, then either $ E\leq Z_{\UU_{p}}(G) $ or $ |E|_{p}=p $.
    \end{theorem}

    \begin{theorem}[{\cite[Proposition 1.6]{Chen-2013}}]\label{minamal}
    	Let $ E $ be a normal subgroup of $ G $ and let $ P $ be a Sylow $ p $-subgroup of $ E $. If  every cyclic subgroup of $ P $ of prime order or order $ 4 $ (when $ P $ is
    	not quaternion-free)  satisfies the partial $ \Pi $-property in $ G $, then  $ E\leq Z_{\UU_{p}}(G) $.
    \end{theorem}

    %In this note, we study  the influence of second minimal or second maximal subgroups of Sylow  subgroups  satisfying the partial $ \Pi $-property  on the structure of finite groups.

    In this note, we investigate the structure of the groups in which some second minimal or second maximal  subgroups of a Sylow subgroup  satisfy  the partial $ \Pi $-property. Our  results  are as follows:

    \begin{theorem}\label{2-mini}
    Let $ P $ be a Sylow $ p $-subgroup of $ G $ with $ |P|\geq  p^{2} $. Assume that every $ 2 $-minimal subgroup of  $ P $  satisfies the partial $ \Pi $-property in $ G $.  Then $ G $ is $ p $-soluble with  $ p $-length at most $ 1 $.
    \end{theorem}

    \begin{theorem}\label{2-maxi}
     Let $ P $ be a Sylow $ p $-subgroup of $ G $ with $ |P|\geq  p^{3} $. Assume that every $ 2 $-maximal subgroup of $ P $ satisfies the partial $ \Pi $-property in $ G $ and  every cyclic subgroup of $ P $ of order  $ 4 $ satisfies the partial $ \Pi $-property in $ G $ when $ P $ is
     isomorphic to $ Q_{8} $.  Then $ G $ is $ p $-soluble with   $ p $-length  at most $ 1 $.
    \end{theorem}

    In fact, we can not obtain the $ p $-supersolubility of $ G $ in Theorem \ref{2-mini} or Theorem \ref{2-maxi}. Let us see the following example.

    \begin{example}
    	Consider an elementary abelian group $ U=\langle x, y|x^{5}=y^{5}=1, xy=yx \rangle $ of order $ 25 $.  Let $ \alpha $ be an automorphism of $ U $ of order $ 3 $ such that $ x^{\alpha}=y, y^{\alpha}=x^{-1}y^{-1} $. Let  $ V=\langle a, b \rangle $ be a copy of $ U $ and $ G = (U\times V)\rtimes \langle \alpha \rangle $. For any subgroup $ H $ of $ G $ of order $ 25 $, there exists a minimal normal subgroup $ K $ such that $ H\cap K = 1 $ (for details, see \cite[Example]{Guo-Xie-Li}). Then $ G $ has a  chief series $$ \varGamma_{G}: 1 =G_{0} <K= G_{1} < HK=G_{2} < G_{3}= G $$  such that for every $ G $-chief factor $ G_{i}/G_{i-1} $ $(1\leq i\leq 3) $ of $ \varGamma_{G} $, $ | G  : N _{G} (HG_{i-1}\cap G_{i})| $ is a $ 5 $-number. It means that $ H $ satisfies the partial $ \Pi $-property in $ G $. However, $ G $ is not $ 5 $-supersoluble.
    \end{example}

%\vskip1cm
\section{Preliminaries}

  \begin{lemma}[{\cite[Lemma 2.1]{Chen-2013}}]\label{over}
  	Let $ H \leq G $ and $ N \unlhd G $. If either $ N \leq H $ or $ (|H|, |N|)=1 $ and $ H $ satisfies the partial $ \Pi $-property in $ G $, then $ HN/N $ satisfies the partial $ \Pi $-property in $ G/N $.
  \end{lemma}

  \begin{lemma}\label{subgroup}
  	Let  $ H\leq N \leq G $.  If $ H $ is a $ p $-subgroup of $ G $ and $ H $ satisfies the partial $ \Pi $-property in $ G $, then $ H $ satisfies the partial $ \Pi $-property in $ N $.
  	
  \end{lemma}

  \begin{proof}
  	By hypothesis, there exists a chief series $ \varGamma_{G}: 1 =G_{0} < G_{1} < \cdot\cdot\cdot < G_{n}= G $ of $ G $ such that for every $ G $-chief factor $ G_{i}/G_{i-1}$ $ (1\leq i\leq n) $ of $ \varGamma_{G} $, $ | G  : N _{G} (HG_{i-1}\cap G_{i})| $ is a $ p $-number. Therefore, $ \varGamma_{N} : 1=G_{0}\cap N\leq G_{1}\cap N\leq\cdot\cdot\cdot <G_{n}\cap N=N $ is, avoiding repetitions, a normal series of $ N $.  For any  normal section $ (G_{i}\cap N)/(G_{i-1}\cap N) $ $ (1\leq i\leq n) $ of $ N $,   we have that $ H(G_{i-1}\cap N)\cap (G_{i}\cap N)=(H\cap G_{i})G_{i-1}\cap N $. Since $ | G  : N _{G} (HG_{i-1}\cap G_{i})| $ is a $ p $-number, we deduce that $ | N  : N\cap N _{G} (HG_{i-1}\cap G_{i})| $ is a $ p $-number, and so $ | N  : N _{N} (HG_{i-1}\cap G_{i}\cap N)| $ is a $ p $-number.  Since $ H(G_{i-1} \cap N)\cap (G_{i}\cap N)=(H\cap G_{i})G_{i-1}\cap N $, it follows that $ | N  : N _{N} (H(G_{i-1} \cap N)\cap (G_{i}\cap N))| $ is a $ p $-number. Let  $ A/B$ be a chief factor of $ N $ such that $ G_{i-1}\cap N\leq B\leq A\leq G_{i}\cap N $. Then $ N _{N} (H(G_{i-1} \cap N)\cap (G_{i}\cap N))\leq N _{N} ((H(G_{i-1} \cap N)\cap A)B)=N_{N}(HB\cap A) $. Hence $ |N:N_{N}(HB\cap A)| $ is a $ p $-number. This  means that  $ H $ satisfies the partial $ \Pi $-property in $ N $.
  \end{proof}

  \begin{lemma}\label{pass}
  	Let $ H $ be a $ p $-subgroup of $ G $  and $ N $ be  a normal subgroup of $ G $  containing $ H $. If  $ H $ satisfies the partial $ \Pi $-property in $ G $, then $ G $ has  a chief series $$ \varOmega_{G}: 1 =G^{*}_{0} < G^{*}_{1} < \cdot\cdot\cdot <G^{*}_{r}=N < \cdot\cdot\cdot < G^{*}_{n}= G $$  passing through $ N $ such that $ |G:N_{G}(HG^{*}_{i-1}\cap G^{*}_{i})| $ is a $ p $-number  for every $ G $-chief factor $ G^{*}_{i}/G^{*}_{i-1} $ $ (1\leq i\leq n) $ of $ \varOmega_{G} $.
  \end{lemma}

  \begin{proof}
  	Since $ H $ satisfies the partial $ \Pi $-property in  $ G $, there exists a chief series $ \varGamma_{G}: 1 =G_{0} < G_{1} < \cdot\cdot\cdot < G_{n}= G $ of $ G $ such that for every $ G $-chief factor $ G_{i}/G_{i-1} $ $ (1\leq i\leq n) $ of $ \varGamma_{G} $, $ | G  : N _{G} (HG_{i-1}\cap G_{i})| $ is a $ p $-number. Since $ N $ is a normal subgroup of $ G $, we see that $ \varGamma_{G}\cap N: 1 =G_{0}\cap N < G_{1}\cap N < \cdot\cdot\cdot < G_{n}\cap N= N $ is, avoiding repetitions, part of a chief series of $ G $. Note that $ H\leq N $, and so $ |G:N_{G}((H(G_{i-1}\cap N))\cap (G_{i} \cap N))| $ is a $ p $-number.
  	
  	We can complete $ \varGamma_{G}\cap N $ to obtain a chief series $ \varOmega_{G} $ of $ G $.  Then $ G $ has  a chief series $$ \varOmega_{G}: 1 =G^{*}_{0} < G^{*}_{1} < \cdot\cdot\cdot <G^{*}_{r}=N < \cdot\cdot\cdot < G^{*}_{n}= G $$  such that $ |G:N_{G}(HG^{*}_{i-1}\cap G^{*}_{i})| $ is a $ p $-number  for every $ G $-chief factor $ G^{*}_{i}/G^{*}_{i-1} $ $ (1\leq i\leq n) $ of $ \varOmega_{G} $,  as desired.
  \end{proof}

  \begin{lemma}[{\cite[Lemma 6]{Ballester-2011}}]\label{component}
  	 Let $ G $ be a group such that $ O_{p'}(G) = 1 $ for some prime $ p $. Suppose that $ E (G) $ is non-trivial and $ F(G) = {\rm Soc}(G) $. For every component $ H $ of $ G $, we have:
  	
  	 \vskip0.08in
  	
  	 \noindent{\rm (1)} $ p $ divides the order of $ H/Z(H) $, and
  	
  	 \noindent{\rm (2)} the Sylow p-subgroups of $ H/Z(H) $ are non-cyclic.
  \end{lemma}

  \begin{lemma}[{\cite[Proposition 1]{Ballester-Bolinches-1996}}]\label{important}
  	Let $ \FF  $ be a saturated formation and let $ G $ be a group which does not belong to $ \FF $. Suppose that there exists a maximal subgroup $ M $ of $ G $ such that $ M\in \FF  $  and $ G =
  	MF(G) $. Then $ G^{\FF}/(G^{\FF})' $ is a chief factor of $ G $, $ G^{\FF} $ is a $ p $-group for some prime $ p $,
  	$ G^{\FF} $ has exponent $ p $ if $ p $ is odd and exponent at most $ 4 $ if $ p = 2 $. Moreover, either $ G^{\FF} $
  	is elementary abelian or $ (G^{\FF})'= Z(G^{\FF}) = \Phi(G^{\FF}) $ is elementary abelian.

  \end{lemma}

  \begin{lemma}[{\cite[Theorem 7]{Ballester-2011}}]\label{second-maximal}
  	
  	Let $ G $ be a group whose Sylow $ p $-subgroups have order $ p^{2} $. Consider the quotient group $ G^{+}= G /O_{p'}( G ) $ and denote $ S = {\rm Soc}(G^{+}) $ and $ F = O^{p'}( G^{+})  $. Then one of the following holds:
  	
  	\vskip0.08in
  	
  	\noindent{\rm(1)} S is a chief factor of G. In this case, one of the following holds:
  	
  	{\rm(a)} $ S $ is isomorphic to a cyclic group of order $ p $, $ G $ is $ p $-supersoluble and $ F $ is isomorphic to a cyclic group of order $ p^{2} $.
  	
  	{\rm(b)} $ G^{+} $ is a primitive group of type $ 1 $, $ G $ is $ p $-soluble and $ F = S $ is isomorphic to an elementary abelian group of order $ p^{2} $.
  	
  	{\rm(c)} $ G^{+} $ is a primitive group of type $ 2 $ and $ F = S $ and either $ S $ is a product of two copies of a non-abelian simple group whose Sylow $ p $-subgroups have order $ p $ or $ S $ is a simple group whose Sylow $ p $-subgroups have order $ p^{2} $.
  	
  	{\rm(d)} $ G^{+} $ is a primitive group of type $ 2 $ and $ S < F $ and $ G^{+} $ is an almost-simple group such that $ S $ is a non-abelian simple group with Sylow $ p $-subgroups of order $ p $ and $ G^{+}/ S $ is a soluble group with Sylow $ p $-subgroups of order $ p $.
  	
  	\vskip0.08in
  	
  	\noindent{\rm(2)} $ S $ is the direct product of two distinct minimal normal subgroups of $ G^{+} $, $ N_{1} $ and $ N_{2} $ say. In this case, $ F = S $ and $ N_{1} $ and $ N_{2} $ are simple groups with cyclic Sylow $ p $-subgroups of order $ p $.
  \end{lemma}

\begin{lemma}\label{isomorphic}
	Let $ P $ be a Sylow $ p $-subgroup of $ G $ with $ |P|=p^{3} $.
	Suppose that $ G $ is not  $ p $-soluble   and every  $ 2 $-maximal subgroups of $ P $ satisfies  the partial $ \Pi $-property in $ G $. Then $ p = 2 $ and $ P $ is  isomorphic to $ Q_{8} $.
\end{lemma}

\begin{proof}
	By Lemma \ref{over},  we may  assume  that $ O_{p'}(G) = 1 $. Hence, $ F(G) = O_{p}(G)  $. It is no loss to assume that $ O^{p'}(G)=G $. Since $ G $ is not  $ p $-soluble, we have that $ F(G)<P $.  Let $ K $ be a minimal normal subgroup of $ G $. Then $ 1 < P\cap K $. Suppose that $ |P\cap K |\geq p^{2} $ . Then there exists a $ 2 $-maximal normal  subgroup $ H $ of $ P $ such that $  H<P\cap K $. By hypothesis, $ H $ satisfies  the partial $ \Pi $-property in $ G $. There exists a chief series $$ \varGamma_{G}: 1 =G_{0} < G_{1} < \cdot\cdot\cdot < G_{n}= G $$ of $ G $ such that for every $ G $-chief factor $ G_{i}/G_{i-1} $ $ (1\leq i\leq n) $ of $ \varGamma_{G} $, $ | G : N _{G} (HG_{i-1}\cap G_{i})| $ is a $ p $-number. Note that there exists an integer $ j $ $ (1 \leq j \leq n) $ such that $ G_{j}=  G_{j-1} \times K $. It follows that $ | G  : N _{G} (HG_{j-1}\cap G_{j})| $ is a $ p $-number, and so $ HG_{j-1}\cap G_{j}\unlhd G $. This yields that $ HG_{j-1}\cap G_{j}=G_{j-1} $ or $  HG_{j-1}\cap G_{j}=G_{j} $. If $ HG_{j-1}\cap G_{j}=G_{j-1} $, then $ 1<H= H\cap G_{j}\leq G_{j-1} $, a contradiction. If $  HG_{j-1}\cap G_{j}=G_{j} $, then $ G_{j}\leq HG_{j-1} $, which contradicts the fact that  $ G_{j-1}H < G_{j-1}(P\cap K) \leq  G_{j} $. Therefore, $ |P\cap K|=p $. By hypothesis, $ P\cap K $ satisfies  the partial $ \Pi $-property in $ G $. Consider the $ G $-chief factor $ K/1 $. Applying Lemma \ref{pass}, we have that $ |G:N_{G}(P\cap K)| $ is a $ p $-number. It follows that $ P\cap K\unlhd G $, and thus $ P\cap K=K $.  Hence every minimal normal subgroup of $ G $ has order $ p $. In particular,  $ {\rm Soc}(G) \leq  O_{p}(G) = F(G) < P $.
	
	Assume that $ N $ is  a minimal normal subgroup of $ G $ such that $ N<F(G) $. Then $ |N|=p $.  Write $ M/N = O_{p'}(G/N ) $. Then $ P/N $ has order $ p^{2} $, $ O_{p'}(G/M) = 1 $ and $ F(G)M/M $ has order $ p $. By Lemma \ref{second-maximal}, we have that $ {\rm Soc}(G/M) = F(G)M/M \times A/M $, where $ A/M $  is a non-abelian simple
	group whose Sylow $ p $-subgroups have order $ p $. Then $ A $ is a non-$ p $-soluble group with Sylow $ p $-subgroups of order $ p^{2} $. Clearly, $ A $ is normal in $ G $,  $ O_{p'}(A) = 1 $. By Lemma \ref{second-maximal}, we have that $ {\rm Soc}(A) = N \times V $, where $ V $ is a non-abelian simple group with Sylow $ p $-subgroups of order $ p $. As a consequence,  the normal closure $ V^{G} $ of $ V $ in $ G $ is a minimal normal subgroup of $ G $, contrary to the fact that $ {\rm Soc}(G) $ is abelian.  Hence there exists a minimal normal subgroup $ L $ of $ G $ such that $ L={\rm Soc}(G)=O_{p}(G)=F(G) $.  If $ F^{*}(G) = F(G) $, then $ P \leq  C_{G}(F^{*}(G)) \leq F^{*}(G) = L $. This is impossible. Hence $ F(G ) < F^{*}(G) $. It means that $ E(G) $ is non-trivial and then $ L \leq E(G) $. Notice that  $ Z(E(G))>1 $, and we have  $ Z(E(G))=L $. By Lemma \ref{component}, for every component $ R $ of $ G $, $ p $ divides  $ |RN/N| $ and the Sylow
	$ p $-subgroups of $ RN/N $ are non-cyclic. Note that $ RN\unlhd \unlhd G $,  and hence a Sylow $ p $-subgroup $ P/N $ of $ RN/N $ is elementary abelian of order $ p^{2} $. In particular, $ P $ is not cyclic.  Let $ U $ be a maximal normal subgroup of $ G $. If $ |G/U|_{p}=p $, then $ U $ is $ p $-supersoluble by Theorem \ref{maximal}. Since $ O_{p'}(G)=1 $, it follows from \cite[Lemma 2.1.6]{Ballester-2010} that $ P\cap U\unlhd U $, and thus $ P\cap U\unlhd G $. Hence $ G $ is $ p $-soluble, a contradiction. Therefore, $ G/U $ is a non-abelian simple group whose Sylow $ p $-subgroups have  order $ p^{2} $. In particular, $ U $ is $ p $-soluble. Let $ X $ be a normal maximal subgroup of $ G $ which is different from $ U $. With a similar argument, we  can get that $ |G/X|_{p}=p^{2} $  and $ X $ is $ p $-soluble. Then $ G=UX $ is $ p $-soluble, a contradiction.  This shows that $ U $ is  the unique maximal normal subgroup of $ G $. Then  the chief factor $ G/U $ occurs in every chief series of $ G $. Assume that $ p>2 $.  If  $ L $ is the unique  $ 2 $-maximal subgroup of $ P $, then $ P $ is cyclic by \cite[Kapitel III, Satz 8.3]{Huppert-1967}, a contradiction. Hence there exists a $ 2 $-maximal subgroup $ T $ of $ P $ such that $ L\not =T $. Then $ T\nleq U $. By hypothesis, $ T $ satisfies  the partial $ \Pi $-property in $ G $. Consider the $ G $-chief factor $ G/U $.  Then $ |G:N_{G}(UT\cap G)|=|G:N_{G}(UT)| $ is a $ p $-number, and thus $ G=PN_{G}(UT) $.  Observe that $ (P\cap U)T=LT $ is a maximal subgroup of $ P $. Then $ UT=U(P\cap U)T $ is normalized by $ P $. It follows that $ U<UT\unlhd G $, and hence $ |G/U|=p $,  a contradiction. Therefore, we can assume that $ p=2 $. If $ P $ has a $ 2 $-maixmal subgroup $ S $ such that $ S\not =L  $, then $ |G/U|=p $ by a similar argument as above. This is a contradiction.  Hence $ L $ is the unique $ 2 $-maximal subgroup of $ P $. Since $ P $ is non-cyclic, it follows from  \cite[Kapitel III, Satz 8.2]{Huppert-1967} that $ P $ is  isomorphic to $ Q_{8} $, as desired. Our proof is now complete.
\end{proof}

  \begin{lemma}\label{order}
  	Let $ P $ be a Sylow  $ p $-subgroup of $ G $. If  $ P $ satisfies the partial $ \Pi  $-property in $ G $, then $ G $ is $ p $-soluble.
  \end{lemma}

  \begin{proof}
  	By hypothesis, $ P $ satisfies  the partial $ \Pi $-property in $ G $. There exists a chief series $$ \varGamma_{G}: 1 =G_{0} < G_{1} < \cdot\cdot\cdot < G_{n}= G $$ of $ G $ such that for every $ G $-chief factor $ G_{i}/G_{i-1} $ $ (1\leq i\leq n) $ of $ \varGamma_{G} $, $ | G/G_{i-1} : N _{G/G_{i-1}} (PG_{i-1}/G_{i-1}\cap G_{i}/G_{i-1})| $ is a $ p $-number. Then $ | G_{i}/G_{i-1} : N _{G_{i}/G_{i-1}} ((P\cap G_{i})/G_{i-1})| $ is a $ p $-number. Note that $ (P\cap G_{i})G_{i-1}/G_{i-1} $ is a Sylow $ p $-subgroup of $ G_{i}/G_{i-1} $, and so $ G_{i}/G_{i-1} $ is $ p $-soluble. Thus $ G $ is $ p $-soluble.
  \end{proof}

  If $ P $ is either an odd order $ p $-group or a quaternion-free $ 2 $-group, then we use $ \Omega(P) $ to denote the subgroup $ \Omega_{1} (P) $.
  Otherwise, $ \Omega (P) = \Omega_{2} (P) $.

  \begin{lemma}\label{hypercenter}
  	Let $ P $ be a normal $ p $-subgroup of $ G $ and $ D $ a Thompson critical subgroup of $ P $ \cite[page 186]{Gorenstein-1980}. If $ P/\Phi(P) \leq Z_{\UU}(G/\Phi(P)) $ or  $ \Omega(D) \leq Z_{\UU}(G) $, then $ P \leq  Z_{\UU}(G) $.
  \end{lemma}

  \begin{proof}
  	By \cite[Lemma 2.8]{Chen-xiaoyu-2016}, the conclusion follows.
  \end{proof}

  %\begin{lemma}[{\cite[Lemma 2.11]{Chen-xiaoyu-2014}}]\label{nilpotent-class}
  	 %Let $ P $ be a $ p $-group of nilpotent class at most $ 2 $. Suppose that the exponent of $ P/Z(P) $ divides $ p $.
  	
  	 %\vskip0.08in
  	
  	%\noindent{\rm (1)} If $ p $ is odd, then the exponent of $ \Omega(P) $ is $ p $.
  	
  	%\noindent{\rm (2)} If $ P $ is a non-abelian $ 2 $-group, then the exponent of $ \Omega(P) $ is $ 4 $.
  %\end{lemma}

 \begin{lemma}[{\cite[Lemma 2.10]{Chen-xiaoyu-2016}}]\label{critical}
 	Let $ D $ be a Thompson critical subgroup of a non-trivial $ p $-group $ P $.
 	
 	\vskip0.08in
 	
 	\noindent{\rm (1)} If $ p > 2 $, then the exponent of $ \Omega_{1}(D) $ is $ p $.
 	
 	\noindent{\rm (2)} If $ P $ is an abelian $ 2 $-group, then the exponent of $ \Omega_{1}(D) $ is $ 2 $.
 	
 	\noindent{\rm (3)} If $ p = 2 $, then the exponent of $ \Omega_{2}(D) $ is at most $ 4 $.
 \end{lemma}

  \begin{lemma}[{\cite[Lemma 3.1]{Ward}}]\label{charcteristic}
  	Let $ P $ be a non-abelian quaternion-free $ 2 $-group. Then $ P $ has a characteristic subgroup of index $ 2 $.
  \end{lemma}

  \begin{lemma}[{\cite[page 26]{Berkovich}}]\label{non-abelian}
  	A non-abelian $ p $-group $ G $ has only one normal subgroup of index $ p^{2} $ if and only if
  	$ G'=\Phi(G) $ is of index $ p^{2} $ in $ G $. In particular, a non-abelian $ 2 $-group has only one
  	normal subgroup of index $ 4 $ if and only if it is dihedral, semidihedral or generalized quaternion.
  \end{lemma}

  \begin{lemma}[{\cite[Lemma 2.6(i)]{guo2017conditions}}]\label{N-cyclic}
  	If a group  $ P $ is dihedral, semidihedral or generalized quaternion, then $ P $ has  only  one  normal subgroup $ N $ of order $ 2^{n} $ for every $ 1<2^{n}<|P|/2 $ and $ N $ is cyclic, where $ n $ is a positive integer.
  \end{lemma}

 \begin{lemma}\label{if-only}
	Suppose that $ P $ is a normal $ p $-subgroup of $ G $. Then $ P $ is contained in the  hypercenter $ Z_{\infty}(G) $ of  $ G $  if and only if $ O^{p}(G) \leq C_{G}(P) $.
\end{lemma}

\begin{proof}
	If $ P\leq Z_{\infty}(G) $, then $ C_{G}(L/K)=G $ for any chief factor $ L/K $ of $ G $  below $ P $.  In particular, every $ p' $-subgroup of $ G $ centralizes $ L/K $. By \cite[Chapter 5, Theorem 3.2]{Gorenstein-1980}, every $ p' $-subgroup of $ G $ centralizes $ P $,  and thus  $ O^{p}(G)\leq C_{G}(P) $.

	Conversely, if $ O^{p}(G) \leq C_{G}(P) $, then $ |G/C_{G}(P)| $ is a power of $ p $. By \cite[page 220, Theorem 6.3]{Deskins-1982}, we have $ P\leq Z_{\infty}(G) $.
\end{proof}

  %\begin{lemma}[{\cite[Kapitel VI, Satz 14.3]{Huppert-1967}}]\label{identity}
  	%Suppose that $ G $ has an abelian Sylow $ p $-subgroup $ P $. Then $ G' \cap Z(G)\cap P=1 $.
  %\end{lemma}

  %\begin{lemma}[{\cite[Theorem 5.18]{isaacs2008finite}}]\label{P-abelian}
  	%Let $ P $ be an abelian Sylow $ p $-subgroup of $ G $. Then $ G' \cap P\cap Z(N_{G}(P))=1 $.
  %\end{lemma}

  \section{Proofs}

  In order to prove Theorem \ref{2-mini} and Theorem \ref{2-maxi}, we shall prove a series of results at first.

  \begin{theorem}\label{2-minimal}
  	Let $ P $ be a Sylow $ p $-subgroup of a  $ p $-soluble group $ G $.  Assume that every $ 2 $-minimal subgroup of $ P $ satisfies the partial $ \Pi $-property in $ G $.  Then the $ p $-length of $ G $ is at most $ 1 $.
  \end{theorem}

  \begin{proof}
  	Assume  that the theorem is not true and $ G $ is a counterexample of  minimal  order. Let $ \FF $ denote the class of all $ p $-soluble groups whose $ p $-length is at most $ 1 $. Then  $ \FF $ is a saturated formation by \cite[Chapter IV, Examples 3.4]{Doerk-Hawkes}. We divide the proof into the following steps.
  	
  	\vskip0.1in
  	
  	\noindent\textbf{Step 1.} $ O_{p'}(G) = 1 $ and $ O^{p'}(G) = G $.
  	
  	\vskip0.1in
  	
  	Applying Lemma \ref{over}, $ G/O_{p'}(G) $ satisfies  the hypotheses of the theorem. If $ |G/O_{p'}| < |G| $, it  follows that $  G/O_{p'} \in \FF $, and thus $ G \in \FF $. This is a contradiction. Hence $ O_{p'}(G)=1 $.
  	
  	Assume that $ O^{p'}(G) < G $. Since  $ O^{p'}(G) $ satisfies the hypotheses of the  theorem by Lemma \ref{subgroup}, we have  that $ O^{p'}(G)\in \FF $. Thus $ G\in \FF $, a contradiction. Therefore $ O^{p'}(G) = G $.
  	
  	\vskip0.1in
  	
  	\noindent\textbf{Step 2.} Any proper subgroup of $ G $ belongs to $ \FF $.
  	
  	\vskip0.1in
  	
  	Let $ R $ be a proper subgroup of $ G $. If a Sylow $ p $-subgroup  of $ R $ is of order at most $ p $, then  $ R\in \FF $ since $ G $ is $ p $-soluble. Assume that $ p^{2} $ divides $ |R| $ and let $ L $ be a subgroup of $ R $ of order
  	$ p^{2} $. Then $ L $ satisfies the partial $ \Pi $-property in $ R $ by Lemma \ref{subgroup}, and so $ R $  satisfies  the hypotheses of the theorem. The minimal choice  of $ G $ yields that  $ R \in \FF  $. Hence  Step 2 holds.
  	
  	\vskip0.1in
  	
  	\noindent\textbf{Step 3.} There exists a maximal subgroup $ M $ of $ G $ such that $ M \in \FF $ and $ G = MG^{\FF} $. Moreover, $ G^{\FF} \cap M\leq \Phi(G) $, $ G^{\FF} $ is a $ p $-group and $ G^{\FF} /\Phi( G^{\FF})=G^{\FF} /( G^{\FF})' $ is a chief factor of $ G $.
  	
  	\vskip0.1in
  	
  	Since $ G \not \in \FF $ and $ \FF $ is saturated, it follows that $ G /\Phi(G) $ does not belong to $ \FF $.   Let $ N/\Phi(G) $ be  a minimal normal subgroup of $ G /\Phi(G)  $. Then  there exists  a maximal subgroup $ M $ of $ G $ such that $ G=MN $. Since $ G $ is $ p $-soluble, we get that  $ N<G $.   By Step 2, $ G/N $ belongs to $ \FF $.  Since $ \FF $ is a formation, it follows that $ N/\Phi(G) $ is  a unique minimal normal subgroup of  $ G /\Phi(G) $. Then $ N=G^{\FF}\Phi(G) $ and $ G = MG^{\FF} $. Note that  $ N/\Phi(G) $ is  not a $ p' $-group.  Consequently, $ N/\Phi(G) $ is an elementary abelian $ p $-group, and thus $ N\cap M=\Phi(G) $. This yields that $ G^{\FF} \cap M\leq \Phi(G) $.  By Step 1, we see that $ \Phi(G) $ is a $ p $-group. Hence $ G^{\FF}\leq N\leq O_{p}(G)=F(G) $. As a consequence, $ G=MN=MF(G) $. By  Lemma \ref{important}, we deduce  that $ G^{\FF}/\Phi(G^{\FF})=G^{\FF}/(G^{\FF})' $ is a chief factor of $ G $.

  	%$ N /\Phi(G) $ is supplemented in $ G /\Phi(G) $.  By Step 2, $ G/N $ belongs to $ \FF $. Since $ \FF $ is a formation, $ G /\Phi(G) $ has a unique minimal normal subgroup, $ T /\Phi(G) $ say. Then $ T=G^{\FF}\Phi(G) $ and there is a maximal subgroup $ M $ of $ G $ such that $ G=TM $ and $ T\cap M=\Phi(G) $.  Since $ G $ is $ p $-soluble, it follows that $ T /\Phi(G) $ is an elementary abelian $ p $-group. Note that $ \Phi(G) $ is a $ p $-group by Step 1. Hence $ T\leq F(G) $. Then $ G=MT=MG^{\FF} $.  By  Lemma \ref{important}, we get that $ G^{\FF}/(G^{\FF})'=G^{\FF}/\Phi(G^{\FF}) $ is a chief factor of $ G $, $ G^{\FF} $ is a $ p $-group and  $ G^{\FF} $ has exponent $ p $ if $ p $ is odd and exponent at most $ 4 $ if $ p = 2 $.

  	\vskip0.1in
  	
  	\noindent\textbf{Step 4.} Assume that $ |G^{\FF} |\geq p^{2} $ and
  	$ H $ is a subgroup of $ G^{\FF} $ of order $ p^{2} $. If  $ H\Phi(G^{\FF}) \unlhd P $,  then $ H\Phi(G^{\FF})\unlhd G $.
  	
  	\vskip0.1in
  	
  	By hypothesis, $ H $ satisfies partial $ \Pi $-property in $ G $. By Lemma \ref{pass}, $ G $ has  a chief series
  	$$ \varOmega_{G}: 1 =G^{*}_{0} < G^{*}_{1} < \cdot\cdot\cdot <G^{*}_{r-1} <G^{*}_{r}=G^{\FF} < \cdot\cdot\cdot < G^{*}_{n}= G $$
  	such that $ |G:N_{G}(HG^{*}_{i-1}\cap G^{*}_{i})| $ is a $ p $-number  for every $ G $-chief factor $ G^{*}_{i}/G^{*}_{i-1} $ $ (1\leq i\leq n) $ of $ \varOmega_{G} $.   Since $ G^{\FF} /\Phi( G^{\FF}) $ is a chief factor of $ G $ by Step 3, it follows that either $ G^{*}_{r-1}= \Phi(G^{\FF}) $ or  $ G^{*}_{r-1}\Phi(G^{\FF})=G^{\FF} $. If  $ G^{*}_{r-1}\Phi(G^{\FF})=G^{\FF} $, then $ G^{*}_{r-1}=G^{\FF} $, a contradiction. Hence $ G^{*}_{r-1}= \Phi(G^{\FF}) $. This yields that  $ G^{\FF} /\Phi( G^{\FF}) $ is a $ G $-chief factor of  $ \varOmega_{G} $. Thus  $ |G:N_{G}(H\Phi(G)\cap G^{\FF})| $ is a $ p $-number. Since $ H\Phi(G)\unlhd P $, we have that $ H\Phi(G)\unlhd G $.

  	\vskip0.1in
  	
  	\noindent\textbf{Step 5.} $ \Phi(G^{\FF}) = 1 $, $ |G^{\FF}| = p^{2} $ and $ G^{\FF} $ is a minimal normal subgroup of $ G $.
  	
  	\vskip0.1in
  	
  	Suppose that $ G^{\FF}/\Phi(G^{\FF}) $ has order $ p $. Then $ G^{\FF} $ is a cyclic $ p $-group by Step 3. Thus $ G^{\FF}\leq Z_{\FF}(G) $. By \cite[Theorem 1(iv)]{Ballester-Bolinches-1996}, we have $ Z_{\FF}(G)=\Phi(G) $, and thus $ G^{\FF}\leq \Phi(G) $. This contradicts Step 3.    Therefore $ |G^{\FF}/\Phi(G^{\FF})|\geq p^{2} $.

  	Assume that $ \Phi(G^{\FF}) >1 $.  By Lemma \ref{important}, we see that  $ \Phi(G^{\FF}) $ must be  elementary abelian. Let $ L $ be a normal subgroup of $ P $ contained in $ G^{\FF} $ such that $ |L/\Phi(G^{\FF})|=p $. Pick a  non-identity   element $ x\in L - \Phi(G^{\FF}) $, then $ o(x)=p $ or $ o(x)=p^{2} $. If $ o(x)=p^{2} $, then by Step 4, we know that $ L=\langle x \rangle \Phi(G^{\FF})\unlhd G $, a  contradiction. Hence $ o(x)=p $.  Pick an element $ y\in Z(P)\cap \Phi(G^{\FF}) $ with $  o(y)=p $.  Then $ \langle x \rangle \langle y \rangle $ has order $ p^{2} $. By Step 4, it follows that $ L=\langle x \rangle \langle y \rangle \Phi(G^{\FF})\unlhd G $, also a contradiction. This forces that   $ \Phi(G^{\FF})=1 $. Then $ G^{\FF} $ is elementary abelian. Let $ L $ be a normal subgroup of $ P $ of order $ p^{2} $ contained in $ G^{\FF} $. By Step 4, we have that $ L\unlhd G $. Notice that $ G^{\FF}/1 $ is a chief factor of $ G $, and so $ G^{\FF}=L $. Hence $ |G^{\FF}| = p^{2} $.

  	%$ \langle x \rangle  $ satisfies the partial $ \Pi $-property in $ G $. Thus  $ |G:N_{G}(\langle x \rangle \Phi(G^{\FF})\cap G^{\FF})|=|G:N_{G}( H\Phi(G^{\FF}))| $ is a $ p $-number. This implies that $ H\Phi(G)\unlhd G $, a contradiction. Hence we may assume that $ o(x)=p $. Let $ T $ be a minimal normal subgroup of
  	
  	\vskip0.1in
  	
  	\noindent\textbf{Step 6.}  $ {\rm Core}_{G}(M) > 1 $.
  	
  	\vskip0.1in
  	
  	Assume that $ {\rm Core}_{G}(M)=1 $. Then by Step 3 and Step 5,  $ G^{\FF} $ is the unique minimal normal subgroup of $ G $ and $ |G^{\FF}|=p^{2} $.  Since $ G $ has $ p $-length greater than $ 1 $, we deduce that $ p $ divides $ |M| $.  Choose an element $ u\in G^{\FF}\cap Z(P) $ with $ o(u)=p $. Pick  an element $ v\in M $ such that $ v\not \in G^{\FF} $ and $ o(v)=p $. There exists  an element $ g\in G $  such that $ v^{g}\in P $. Set $ w=v^{g} $. Then $\langle u \rangle \langle w \rangle$ has order $ p^{2} $. By hypothesis, $\langle u \rangle \langle w \rangle$ satisfies the partial $ \Pi $-property in $ G $. Hence $ |G:N_{G}(\langle u \rangle \langle w \rangle \cap G^{\FF})|=|G:N_{G}(\langle u \rangle)| $ is a $ p $-number. This  implies that $ \langle u \rangle \unlhd G $,  a  contradiction.
  	
  	\vskip0.1in
  	
  	\noindent\textbf{Step 7.}   The final contradiction.
  	
  	\vskip0.1in
  	
  	By Step 6, we can pick   a minimal normal subgroup $ T $ of $ G $ contained in $ {\rm Core}_{G}(M) $. Since $ G $ is $ p $-soluble, it follows from Step 1 that  $ T $ is a $ p $-group.  Clearly, $ T\cap G^{\FF}=1 $.  Choose an element $ a\in G^{\FF}\cap Z(P) $ and an element $ b\in T\cap Z(P) $ with  $ o(a)=o(b)=p $.  By hypothesis, $ \langle a , b \rangle $ satisfies the partial $ \Pi $-property in $ G $. Applying  Lemma \ref{pass}, $ G $ has  a chief series $$ \varOmega_{G}: 1 =G^{*}_{0} < G^{*}_{1} < G^{*}_{2}=G^{\FF}T <  \cdot\cdot\cdot < G^{*}_{n}= G $$
  	such that $ |G:N_{G}(\langle a, b \rangle G^{*}_{i-1}\cap G^{*}_{i})| $ is a $ p $-number  for every $ G $-chief factor $ G^{*}_{i}/G^{*}_{i-1} $ $ (1\leq i\leq n) $ of $ \varOmega_{G} $. If $ G^{\FF}=G^{*}_{1} $, then $ |G:N_{G}(\langle a, b \rangle\cap G^{\FF})|=|G:N_{G}(\langle a \rangle)| $ is a $ p $-number. Thus $ \langle a \rangle \unlhd G $, a contradiction. Therefore, $ G^{\FF}\not = G^{*}_{1} $ and $ G^{\FF}\cap  G^{*}_{1}=1 $. Moreover,   $ G^{*}_{1}/1 $ is an $ \FF $-central chief factor of $ G $.  Applying  \cite[Theorem 4.2.17 and Theorem 4.2.4]{Ballester-2006}, we  know  that $ M $ covers  the $ \FF $-central chief factor $ G^{*}_{1}/1 $, i.e.,  $ G^{*}_{1} \leq  M $.  By  Step 3, we see that $ G^{\FF}\cap M\leq \Phi(G) $ and $ G^{\FF}\nleq \Phi(G) $. Note that  $ G^{\FF} $ is a minimal normal subgroup  of $ G $. Then  $ G^{\FF}\cap M\leq \Phi(G)\cap G^{\FF}=1 $.

  	For  the chief factor $ G^{*}_{1}/1 $, we have that  $ |G:N_{G}(\langle a, b \rangle \cap G^{*}_{1})| $ is a $ p $-number, and thus $  \langle a, b \rangle \cap G^{*}_{1}\unlhd G $. Assume that  $ \langle a, b \rangle \cap G^{*}_{1}=1 $. For the chief factor $ G^{\FF}T/G^{*}_{1} $, we know  that $ |G:N_{G}(\langle a, b \rangle G^{*}_{1} \cap G^{\FF}T)|=|G:N_{G}(\langle a, b \rangle G^{*}_{1})| $ is a $ p $-number. Hence $ G^{*}_{1}< \langle a, b \rangle G^{*}_{1}\unlhd G $, and so $ \langle a, b \rangle G^{*}_{1}=G^{\FF}T $. It follows that $ G^{\FF}= G^{\FF}\cap \langle a , b \rangle G^{*}_{1}= \langle a\rangle ( G^{\FF} \cap \langle b \rangle G^{*}_{1})\leq \langle a\rangle ( G^{\FF} \cap M)  = \langle a \rangle $, contrary to Step 5.
  	%Using the $ G $-isomorphism $ \langle a, b \rangle G^{*}_{1}/G^{*}_{1}\cong \langle a, b \rangle /1 $, we see that $ \langle a, b \rangle \unlhd G $. This forces that $ \langle a \rangle =\langle a, b \rangle \cap G^{\FF}\unlhd G $, a contradiction.

  	Therefore, we may assume that $ \langle a, b \rangle \cap G^{*}_{1}=G^{*}_{1} $. By Step 3 and Step 5, $ G^{\FF}\cap M=1 $.  Since $ b\in T\leq M $ and $ G^{*}_{1}\leq M $, we deduce that  $ \langle b \rangle =G^{*}_{1} $. Now consider the chief factor $ G^{\FF}T/ \langle b \rangle $. Then $ |G:N_{G}(\langle a, b \rangle \langle b \rangle \cap G^{\FF}T)| $ is a $ p $-number, and so $  \langle a, b \rangle=\langle a, b \rangle \cap G^{\FF}T \unlhd G $. This  contradicts the fact that $ G^{\FF}T/ \langle b \rangle $ is a chief factor of $ G $. Our proof is now complete.
  \end{proof}

  \begin{theorem}\label{2-maximal}
  	Let $ P $ be a Sylow $ p $-subgroup of a  $ p $-soluble group $ G $.  Assume that every $ 2 $-maximal subgroup of $ P $ satisfies the partial $ \Pi $-property in $ G $.  Then the $ p $-length of $ G $ is at most $ 1 $.
  \end{theorem}

  \begin{proof}
  	Assume that $ G $ does not have $ p $-length at most $ 1 $, and we choose $ G $ of minimum possible  order. Notice that the class of all $ p $-soluble groups of $ p $-length at most $ 1 $ is a saturated formation (see \cite[Chapter IV, Examples 3.4]{Doerk-Hawkes}). By  Lemma \ref{over}, $ G/O_{p'}(G) $ satisfies the  hypotheses of the theorem, so we can assume  that $ O_{p'}(G)=1 $. Let $ N $ be a minimal normal subgroup of $ G $.  Then $ N $ is a $ p $-group since $ G $ is $ p $-soluble. By Lemma \ref{over}, the hypotheses are inherited by $ G/N $. Hence $ G/N $ has $ p $-length of at most $ 1 $. Moreover,  $ N={\rm Soc}(G) $ and $ \Phi(G)=1 $.   Then $ G = N\rtimes K $ for some core-free maximal subgroup $ K $ of $ G $.  Clearly, $ P = N( K \cap P ) $ and $ C_{G}( N ) = N $. Assume that $ | N | \geq p^{3}  $. Let $ N_{1} $ be a normal subgroup of $ P $ contained in $ N $ such  that $ |N/N_{1}|=p^{2} $.  Then $ N_{1}(K\cap P) $ is a $ 2 $-maximal subgroup of $ P $. By hypothesis, $ N_{1}(K\cap P) $ satisfies the partial $ \Pi $-property in $ G $. Hence $ |G:N_{G}(N_{1}(K\cap P)\cap N)| =|G:N_{G}(N_{1})| $ is a $ p $-number. This implies that $ 1<N_{1}\unlhd G $, a contradiction. Therefore $ | N | \leq p^{2} $.  Obviously,  $ N $ is an elementary abelian $ p $-group of order $ p^{2} $ . Then $ K\cong G/C_{G}(N) $ is a subgroup of $ GL(2, p) $. Since $ | GL(2, p)| = (p^{2} -1)(p-1)p $, it follows that $ | P | \leq p^{3}  $. Assume that $ | P | = p^{3} $. Let $ N_{2} $ be a normal subgroup of $ P $ of order  $ p $ contained in $ N $. Then $ N_{2} $ is a $ 2 $-maximal subgroup of $ P $. With a similar argument, we can get that $ N_{2}\unlhd G $, also a contradiction. Therefore $ | P | = p^{2} $ and $ P = N $. It means that the $ p $-length of $ G $ is $ 1 $, the final contradiction.
  \end{proof}

  \begin{theorem}\label{2minimal}
  	 Let $ P $ be a Sylow $ p $-subgroup of $ G $. If  every $ 2 $-minimal subgroup of $ P $ satisfies the partial $ \Pi  $-property in $ G $, then either $ |P|=p $ or $ G $ is a $ p $-soluble group.
  \end{theorem}

  \begin{proof}
  	Assume  that the theorem is not ture and $ G $ is a counterexample of  minimal  order. Then $ |P|\geq p^{2} $ and $ G $ is not $ p $-soluble.  Let $ \FF $ denote the class of all $ p $-soluble groups whose $ p $-length is at most $ 1 $. Applying Lemma \ref{over}, $ G/O_{p'}(G) $  satisfies  the hypotheses of the theorem, so we can assume  that $ O_{p'}(G)=1 $. We divide the proof into the following steps.
  	
  	\vskip0.1in
  	
  	\noindent\textbf{Step 1.}  If $ K $ is a proper subgroup of $ G $ and $ p^{2} $ divides $ | K |  $, then $ K $ is $ p $-soluble. Furthermore, if $ K $ is normal in $ G $,
  	then a Sylow $ p $-subgroup of $ K $ is also normal in $ G $.
  	
  	\vskip0.1in
  	
  	Assume that  $ K $ is a proper subgroup of $ G $ such that $ p^{2} $ divides $ | K |  $, and let $ L $ be a subgroup of $ K $ of order $ p^{2}  $. Then by Lemma \ref{subgroup}, $ L $ satisfies the partial $ \Pi  $-property in $ G $. The minimal choice of $ G $ implies that $ K $ is $ p $-soluble. Suppose that $ K $ is normal in
  	$ G $. Since $ O_{p'}(K) \leq O_{p'}( G ) = 1 $, we conclude that $ O_{p}(K) $ is a Sylow $ p $-subgroup of $ K $ by Theorem  \ref{2-minimal}.
  	
  	\vskip0.1in
  	
  	\noindent\textbf{Step 2.}  $  G $ has a unique maximal normal subgroup, $ N $ say.
  	
  	\vskip0.1in
  	
  	Assume that $ G $ has two different maximal normal subgroups, $ N $ and $ T $ say. Then $ G = NT $. If the
  	Sylow $ p $-subgroups of $ N $ and $ T $ are normal in $ G $, then $ G $ has a normal Sylow $ p $-subgroup and so $ G $ is $ p $-soluble, which is  against  the choice of $ G $.  By  Step 1, we may assume that the order of the Sylow $ p $-subgroups of $ N $ is at most $ p $. If the order of the Sylow $ p $-subgroups of $ T $ is also at most $ p $, then $ |P|\leq p^{2} $. Furthermore, $ |P|= p^{2} $.  By Lemma \ref{order}, we get that $ G $ is $ p $-soluble, a contradiction. Hence, we may assume that $ p^{2} $ divides $ |T| $. Then  a Sylow $ p $-subgroup $ T_{p} $ of $ T $ is normal in $ G $ by Step 1.  In
  	that case $ T_{p}N $ is a normal subgroup of $ G $ containing $ N $. Then $ T_{p}N=G $. By Lemma \ref{order}, we may assume that $ |P|\geq p^{3} $. Then there exists a normal subgroup $ L $ of $ G $ containing $ N $ such that $ |G:L|=p $. Then $ p^{2} $ divides $ L $. By Step 1, $ L $ is $ p $-soluble. Thus $ G $ is $ p $-soluble, a contradiction.
  	
  	\vskip0.1in
  	
  	\noindent\textbf{Step 3.}  $ P\cap N $ is a  normal subgroup of $ G $ and $ N $ is $ p $-soluble. Furthermore, $ G =O^{p}(G) $.
  	
  	\vskip0.1in
  	
  	Assume that $ |P\cap N|\leq p $. Let $ H $ be a normal subgroup of
  	$ P $ of order $ p^{2} $ containing $ P\cap N $. Then $ H $ satisfies the partial $ \Pi  $-property in $ G $. By Step 2, the chief factor $ G/N $  occurs in every chief series of $ G $.  For the chief factor $ G/N $, we have that  $ |G:N_{G}(HN\cap G)| $  is a $ p $-number. It follows that $ HN\cap G\unlhd  G $. If $ HN=G $, then $ |P|=p^{2} $. By Lemma \ref{order}, $ G $ is $ p $-soluble, a contradiction. If $ HN\cap G=N $, then $ H\leq N $,  a contradiction. Therefore, we may assume that $ |P\cap N|\geq p^{2} $.  By Step 1, we conclude that $ P\cap N $ is normal in $ G $ and $ N $ is $ p $-soluble.  If $ O^{p}(G) \not = G $, then $ O^{p}(G)\leq N $ by Step 2. Thus $ O^{p}(G) $ is $ p $-soluble. Since $ G/O^{p}(G) $ is $ p $-group, we have that $ G $ is $ p $-soluble, also a contradiction. Hence $ O^{p}(G)=G $.
  	
  	\vskip0.1in
  	
  	%\noindent\textbf{Step 4.} Every subgroup of $ G $ of order $ p $ or $ p^{2} $ is contained in $ N $.
  	
  	%\vskip0.1in
  	
  	%Let $ R $ be a subgroup of order $ p $ contained in $ N \cap Z(P) $. Assume that $ U $ is a subgroup of order $ p $ which is not contained in $ N $.  Then $ RU $ is a subgroup of $ G $ of order $ p^{2} $. By hypothesis, $ RU $ satisfies the partial $ \Pi  $-property in $ G $. For the chief factor $ G/N $, we have $ |G:N_{G}(UN\cap G)| $  is a $ p $-number. Thus $ UN\cap G\unlhd G $. If $ UN=G $, then $ G/N $ has order $ p $, and so $ G $ is $ p $-soluble by Step 3. This is a contradiction. If $ UN\cap G=N $, then $ U\leq N $, a contradiction.
  	
  	%\vskip0.1in
  	
  	\noindent\textbf{Step 4.} Every subgroup of $ G $ of order $ p $ or $ p^{2} $ is contained in $ N $.
  	
  	\vskip0.1in
  	
  	It is enough to show that every subgroup of $ P $ of order $ p $ or $ p^{2} $ is contained in $ N $. Let $ R $ be a subgroup of order $ p $ contained in $ N \cap Z(P) $. Assume that $ X $ is a subgroup of order $ p $ which is not contained in $ N $.  Then $ RX $ is a subgroup of $ G $ of order $ p^{2} $. By hypothesis, $ RX $ satisfies the partial $ \Pi  $-property in $ G $. For the chief factor $ G/N $, we have $ |G:N_{G}(XN\cap G)|=|G:N_{G}(XN)| $  is a $ p $-number. Thus $ G=N_{G}(XN)P $. Furthermore, $ (XN)^{G}=X^{P}N $. If $ X^{P}N=G $, then $ G/N $ is a $ p $-group, and so $ G $ is $ p $-soluble by Step 3. This is a contradiction. If $ X^{P}N=N $, then $ X\leq N $, a contradiction. Assume that $ Y $ is a subgroup of order $ p^{2} $ which is not contained in $ N $. By hypothesis, $ Y $ satisfies the partial $ \Pi  $-property in $ G $.  Then  we can handle it in a similar way.  Thus Step 4 follows.
  	
  	\vskip0.1in
  	
  	\noindent\textbf{Step 5.} $ N = \Phi(G) = O_{p}(G)=F(G) $.
  	
  	\vskip0.1in
  	
  	Since $ O_{p'}(G) = 1 $, it follows that $ \Phi( G ) $ is a $ p $-group. Suppose  that $ N $ is not contained in $ \Phi( G ) $.  There exists a maximal subgroup $ A $ of $ G $ such that $ G = NA $. By Step 4,  $ G/N $ is a $ p' $-group if $ p^{2} $ did not divide $ |A| $. Then  $ G $ is  $ p $-soluble by Step 3, a  contradiction. This implies that $ p^{2} $ divides $ |A| $. By Step 1, $ A $ is $ p $-soluble. Therefore  $ G/N $ is $ p $-soluble and so is $ G $. This contradiction yields
  	that $ N \leq \Phi( G ) $. Since $ O_{p'}(G)=1 $, we have  $ N = \Phi(G) = O_{p}(G)=F(G) $.
  	
  	\vskip0.1in

  	\vskip0.1in
  	
  	\noindent\textbf{Step 6.}  Every chief factor of $ G $ of order $ p $ is central in $ G $.
  	
  	\vskip0.1in

  	Suppose that $ B/C $ is a chief factor of $ G $ of order $ p $. Then by \cite[Chapter A, Theorem 13.8]{Doerk-Hawkes}, $ N = O_{p}(G) \leq C_{G}(B/C)\leq G $. As a consequence,  either $ C_{G}(B/C) =N $
  	or $ C_{G}(B/C) =G $.  If $ C_{G}(B/C) =N $, then $ G/N=G/C_{G}(B/C)\lesssim {\rm Aut}(B/C) $ is a $ p' $-group. Hence $ G $ is $ p $-soluble, a contradiction. Thus  $ C_{G}(B/C)=G $, i.e., $ B/C $ is central in $ G $.
  	
  	\vskip0.1in
  	
  	\noindent\textbf{Step 7.} Every chief factor of $ G $ below $ N $ does not have order $ p^{2} $.
  	
  	\vskip0.1in
  	
  	Set $ \overline{G}=G/N $.   Let $ L/K $ be a chief factor of $ G $ below $ N $ of order $ p^{2} $. Then $ L/K $ is an elementary
  	abelian $ p $-group and $ L/K $ is an irreducible and faithful $ G/C_{G}(L/K) $-module
  	over the Galois field $ GF(p) $. By \cite[Chapter A, Theorem 13.8]{Doerk-Hawkes} and Step 5, we have  $ N = O_{p}(G)=F(G) \leq C_{G}(L/K) $. Since  $ C_{G}( L/K) \not = G $, it follows that  $ N = C_{G}(L/K) $. Clearly, $ \overline{G} $ can be regarded as a subgroup of $ GL(2, p) $. Since $ \overline{G} $ is a non-abelian simple group, it follows that $ \overline{G}\leq ( GL(2, p))' \leq SL(2, p)  $  and $ \overline{G} \cap Z(SL(2, p)) = 1 $. Therefore,  $ \overline{G} $ can be regarded as a subgroup of $ PSL(2, p) $. According to \cite[Kapitel II, Hauptsatz 8.27]{Huppert-1967}, we know that either $ \overline{G}\cong PSL(2, p) $,
  	or $ \overline{G}\cong  A_{5} $, where $ p = 5 $ or $ p^{2}-1 \equiv 0 $ $ ({\rm mod} $ $ 5) $.   Suppose that  $ \overline{G} $ is  isomorphic to $ PSL(2, p) $. The simplicity  of $ \overline{G} $ implies that $ p\geq 5 $. Since  the index of $ \overline{G} $ in $ SL(2, p) $ is $ 2 $ and $ SL(2, p) $ is perfect, we conclude that $ SL(2, p)=(SL(2, p))'\leq \overline{G} $, a contradiction. If $ \overline{G}\cong A_{5}\cong PSL(2, 5) $, then we can obtain a contradiction with a similar argument. Therefore, Step 7 holds.

    \vskip0.1in
  	
  	\noindent\textbf{Step 8.}  $ N $ is contained in the hypercenter  $ Z_{\infty}(G) $ of $ G $.
  	
  	\vskip0.1in

  	By Step 6, we only need to show that every chief factor of $ G $ below $ N $ has order $ p $. Assume that there exists a chief factor $ U/V $ of $ G $ below $ N $ with $ |U/V|\geq p^{2} $. We choose $ U $ of minimal order. Then $  T\leq Z_{\UU}(G) $ for any normal subgroup $ T $ of $ G $ contained in $ N $ with $ |T|<|U| $.

  	%every chief factor $ T/ K $ of $ G $ below $ N $ with $ |T| < | U | $ has order $ p $.

  	We claim that the exponent of $ U $ is $ p $ or $ 4 $ (when $ U $  is not quaternion-free). If $ U $ is either an odd order $ p $-group or a quaternion-free $ 2 $-group, then we use $ \Omega(U) $ to denote the subgroup $ \Omega_{1} (U) $.
  	Otherwise, $ \Omega (U) = \Omega_{2} (U) $.  Let $ D $ be a Thompson critical subgroup of $ U $.  If $ \Omega(D)< U $, then $ \Omega(D) \leq Z_{\UU}(G) $ by the choice of $ U $.
  	By Lemma \ref{hypercenter}, we have that $ U\leq Z_{\UU}(G) $. This implies that $ |U/V|=p $, a  contradiction. Thus $ U = D = \Omega(D) $.  If $ U $ is a non-abelian quaternion-free $ 2 $-group, then $ U $ has a characteristic subgroup $ R $ of index $ 2 $ by Lemma \ref{charcteristic}. By the choice of $ U $, we have  that $ R \leq Z_{\UU}(G) $, and so $ U\leq Z_{\UU}(G) $, which is impossible. This means that $ U $ is a non-abelian $ 2 $-group if and only if $ U $ is not quaternion-free.  In view of  Lemma \ref{critical}, the exponent of $ U $ is $ p $ or $ 4 $ (when $ U $  is not quaternion-free), as claimed.
  	
  	 %Notice that $ A < P  $ and  $ A\unlhd \widehat{P} $.
  	 Assume that $ V>1 $. There exists  a normal subgroup $ A $ of $ P $  such that $ V< A< U $ and $ |A:V|=p $. Pick a non-identity  element $ a\in A-V $. Then $ \langle a \rangle V=A $ and  $ |\langle a \rangle|\leq p^{2} $.  Let $ L $ be a  subroup of $ A $ with order  $ p^{2} $ such that $ \langle a \rangle \leq L $. Then $ A=LV $.
  	 By hypothesis, $ L $ satisfies the partial $ \Pi  $-property in $ G $. Applying Lemma \ref{pass}, $ G $ has  a chief series
  	 $$ \varOmega_{G}: 1 =G^{*}_{0} < G^{*}_{1} < \cdot\cdot\cdot <G^{*}_{r-1} <G^{*}_{r}=U < \cdot\cdot\cdot < G^{*}_{n}= G $$
  	 such that $ |G:N_{G}(HG^{*}_{i-1}\cap G^{*}_{i})| $ is a $ p $-number  for every $ G $-chief factor $ G^{*}_{i}/G^{*}_{i-1}$ $ (1\leq i\leq n) $ of $ \varOmega_{G} $. Then  every chief factor of $ G $ below $ G^{*}_{r-1} $ has order $ p $ by the choice of $ U $. If $ G^{*}_{r-1}\not =V $, then $ U=VG^{*}_{r-1}\leq Z_{\UU}(G) $, a contradiction.  Therefore $ G^{*}_{r-1} =V $. For  the $ G $-chief factor $ U/V $, we have that $ |G:N_{G}(LV\cap U|=|G:N_{G}(A)| $ is a $ p $-number, and thus $ A\unlhd G $.  It shows that $ A/V $ is a $ G $-chief factor of order $ p $, which  contradicts the fact that $ |U/V|\geq p^{2} $. This means that every chief factor of $ G $ below $ N $ has order $ p $.
  	
  	 Assume that $ V=1 $. By Step 7, we have $ |U/V|\geq p^{3} $. Let $ Q $ be a normal subgroup of $ P $ of order $ p^{2} $ such that $ 1=V<Q<U $. Then $ Q $ satisfies the partial $ \Pi  $-property in $ G $. For the $ G $-chief factor $ U/1 $, $ |G:N_{G}(Q)| $ is a $ p $-number by Lemma \ref{pass}. Hence $ Q\unlhd G $, a contradiction.
  	
  	 %By Step 7, we deduce that $ LV=V $,  and so $ L\leq V\leq Z_{\infty}(G) $ by Step 6. Therefore,  every subgroup of $ U $ of order $ p^{2} $ is contained  in $ Z_{\infty}(G) $.

  	 \vskip0.1in
  	
  	 \noindent\textbf{Step 9.} The final contradiction.
  	
  	 \vskip0.1in

  	 Note that $ O^{p}(G)=G $ by Step 3. It follows from Step 8 and Lemma \ref{if-only} that $ N \leq Z(G) $.  According to  Step 4, every element of $ G $ of order $ p $ or $ p^{2} $ is contained in $ Z( G ) $. Applying \cite[Kapitel IV, Satz 5.5]{Huppert-1967}, we deduce that  $ G $ is $ p $-nilpotent. This final contradiction completes the proof.
  \end{proof}

  \begin{theorem}\label{classification}
  	Let $ P $ be a Sylow $ p $-subgroup of $ G $ with $ |P|\geq p^{2} $. Suppose that every $ 2 $-maximal subgroup
  	of $ P $ satisfies the partial $ \Pi  $-property in $ G $. Then one of the following holds:
  	
  	\vskip0.08in
  	
  	\noindent{\rm (1)} $ G $ is a p-soluble group.
  	
  	\noindent{\rm (2)} $ G $ is a non-$ p $-soluble group and $ |P|= p^{2} $.
  	
  	\noindent{\rm (3)} $ p = 2 $,  $ G $ is a non-$ 2 $-soluble group and $ P $ is isomorphic to $ Q_{8} $.
  \end{theorem}

  \begin{proof}
  	By Lemma \ref{over},  we may assume  that $ O_{p'}(G) = 1 $. It is no loss to assume that $ O^{p'}(G)=G $.  For any  minimal normal subgroup $ K $ of $ G $, we  have $ P\cap K>1 $.
  	
  	Assume that  $ K $ is a minimal normal subgroup of $ G $ such that $ |P:P\cap K|\leq p $. We work to show that $ |P|=p^{2} $ under this assumption. Choose a $ 2 $-maximal subgroup $ H $ of $ P $ such that $ H< P\cap K $ and $ H\unlhd P $. By hypothesis, $ H $ satisfies the partial $ \Pi $-property in $ G $. There exists a chief series $$ \varGamma_{G}: 1 =G_{0} < G_{1} < \cdot\cdot\cdot < G_{n}= G $$ of $ G $ such that for every $ G $-chief factor $ G_{i}/G_{i-1} $ $ (1\leq i\leq n) $ of $ \varGamma_{G} $, $ | G  : N _{G} (HG_{i-1}\cap G_{i})| $ is a $ p $-number. Note
  	that there exists an integer $ j $ $ (1 \leq j \leq n) $ such that $ G_{j}=  G_{j-1} \times K $. Then  $ |G:N_{G}(HG_{j-1}\cap G_{j})| $ is a $ p $-number, and thus $ HG_{j-1}\cap G_{j}\unlhd G $. Since $ HG_{j-1}< (P\cap K)G_{j-1}  \leq  G_{j}  $, it follows that $ HG_{j-1}\cap G_{j}=G_{j-1} $, and hence $ H\cap G_{j}\leq G_{j-1} $.     If $ H>1 $, then $ 1<H=H\cap G_{j}\leq G_{j-1} $, a contradiction. Therefore $ H=1 $. This yields that  $ |P| = p^{2} $. Hence  we are in statement (1) or statement (2).
  	
  	So  in the rest of the proof, we may assume that $ | P :P\cap K |\geq p^{2} $ for every minimal normal subgroup $ K $ of $ G $.  At first, we argue that $ {\rm Soc}(G)\leq P $. Let $ H $ be a $ 2 $-maximal subgroup of $ P $ such that $ 1 < P\cap K\leq H $ and $ H\unlhd P $. By hypothesis, $ H $ satisfies the partial $ \Pi $-property in $ G $. There  exists a chief series
  	$$ \varGamma_{G}: 1 =G_{0} < G_{1} < \cdot\cdot\cdot < G_{n}= G $$ of $ G $ such that for every $ G $-chief factor $ G_{i}/G_{i-1} $ $ (1\leq i\leq n) $ of $ \varGamma_{G} $, $ | G  : N _{G} (HG_{i-1}\cap G_{i})| $ is a $ p $-number. Note
  	that there exists an integer $ j $ $ (1 \leq j \leq n) $ such that $ G_{j}=  G_{j-1} \times K $. It follows that $ | G  : N _{G} (HG_{j-1}\cap G_{j})| $ is a $ p $-number, and so $ HG_{j-1}\cap G_{j}\unlhd G $. This yields that $ HG_{j-1}\cap G_{j}=G_{j-1} $ or $ G_{j}\leq HG_{j-1} $. If $ HG_{j-1}\cap G_{j}=G_{j-1} $, then $ 1<H= H\cap G_{j}\leq G_{j-1} $, a contradiction. Therefore, we have $ G_{j}\leq HG_{j-1} $ and $ K\cong G_{j}/G_{j-1} $ is a $ p $-group. Then  every minimal normal subgroup of $ G $ is a $ p $-group,  and hence  $ {\rm Soc}(G) $ is contained in $ P $, as claimed.

  	Therefore, $ |P/K|=|P/(P\cap K)|\geq p^{2} $. In view of Lemma \ref{over}, $ G/K $ satisfies the hypotheses of the theorem. By induction, $ G /K $ belongs to one of the three statements we described in the theorem. We will complete the proof in the following three cases.

    \vskip0.1in
    \noindent \textbf{Case 1.} $ G/K $ is a $ p $-soluble group.
    \vskip0.1in
  	
  	Since $ K\leq {\rm Soc}(G)\leq P $, we have that $ G $ is $ p $-soluble. Hence $ G $ is of type (1).
  	
   \vskip0.1in
   \noindent \textbf{Case 2.} $ G/K $ is  a non-$ p $-soluble group with a Sylow $ p $-subgroup $ P/K $ of  order $ p^{2} $.
   \vskip0.1in

   %Until now,  we have proved that if $ K $ is a minimal normal subgroup of $ G $, then either $ G/K $ is a $ p $-soluble group or $ G/K $ is  a non-$ p $-soluble group with a Sylow $ p $-subgroup $ P/K $ of  order $ p^{2} $.  If $ G/K $ is a $ p $-soluble group, then $ G $ is $ p $-soluble since $ {\rm Soc}(G)\leq P $. Thus we can assume that for any minimal normal subgroup $ K $,  the quotient group $ G/K $ is a non-$ p $-soluble group and $ |P/K|=p^{2} $.

   In this case,  we  work to  show that $ p = 2 $ and a  Sylow $ 2 $-subgroup $ P $ of $ G $ is  isomorphic to $ Q_{8} $. Note that any minimal normal subgroup of $ G $ is a $ 2 $-maximal subgroup of $ P $.
   Suppose that there exist two distinct minimal normal subgroups $ K  $, $ N $ in $ G $. If $ P = K\times N $, then $ G $ is $ p $-soluble, a contradiction.  Hence $ |P:KN| = p $,  and so $ |P|=p^{3} $. By Lemma \ref{isomorphic}, we have that  $ p = 2 $ and $ P\cong Q_{8} $, as  wanted. Therefore, we may assume that  $ K $ is the unique  minimal normal subgroup of $ G $.  If $ R $ is a  normal $ 2 $-maximal subgroup of $ P $  which is different from $ K $, then $ R\cap K<K $ and $ K<RK $. By hypothesis, $ R $ satisfies the partial $ \Pi  $-property in $ G $.  Consider the $ G $-chief factor $ K/1 $. Then $ |G:N_{G}(R\cap K)| $ is a $ p $-number, and thus $ R\cap K\unlhd G $. The minimality of $ K $ implies that $ R\cap K=1 $. If $ RK=P $, then $ |P|=p^{4} $. By Lemma \ref{2minimal}, $ G $ is $ p $-soluble, a contradiction. If  $ RK $ is a maximal subgroup of $ P $, then $ |P|=p^{3} $. By Lemma \ref{isomorphic}, we have that  $ p = 2 $ and $ P\cong Q_{8} $, as  wanted.

   Therefore, we can assume that $ K $ is the unique normal $ 2 $-maximal subgroup of $ P $. If $ P $ is abelian, then $ P $ is cyclic by \cite[Kapitel III, Satz 8.3]{Huppert-1967}. This forces that $ |K|=p $ and $ |P|=p^{3} $. By Lemma \ref{isomorphic} again, we have that  $ p = 2 $ and $ P\cong Q_{8} $, a contradiction.  If $ P $ is non-abelian and $ p=2 $, then $ P $ is  dihedral, semidihedral or generalized quaternion by Lemma \ref{non-abelian}. Applying Lemma \ref{N-cyclic}, we deduce that $ |K|=2 $, and so $ |P|=8 $.  In view of  Lemma  \ref{isomorphic},  we see that  $ P\cong Q_{8} $, as wanted.  If $ P $ is non-abelian and $ p>2 $, then $ P'=\Phi(P)=K $ by Lemma \ref{non-abelian}.  Let  $ X/K  $ be a chief factor of $ G $. If $ X/K $ is a $ p' $-group,  then $ P\cap X=K=\Phi(P) $. It follows from \cite[Kapitel IV, Satz 4.7]{Huppert-1967} that $ X $ is $ p $-nilpotent. Since  $ O_{p'}(G)=1 $ , we have that  $ X=K $, a contradiction.  Hence $ O_{p'}(G/K)=1 $ and  $ p $  divides the order of $ X/K $. If $ |X/K|_{p}=p $, then every maximal subgroup of $ P\cap X $ satisfies the partial $ \Pi $-property in $ G $. By Theorem \ref{maximal}, we know that $ X\leq Z_{\UU_{p}}(G) $. Therefore  $ X/K $ is a $ G $-chief factor of order $ p $.  By Lemma \ref{second-maximal}, we get  that $ {\rm Soc}(G/K)=X/K\times Y/K $, where $ Y/K $ is a simple group with cyclic Sylow $ p $-subgroups of order $ p $.  By Theorem  \ref{maximal}, we see that $ Y\leq Z_{\UU_{p}}(G) $, and thus  $ |Y/K|=p $. This implies that  $ G $ is $ p $-soluble, a contradiction.  If $ |X/K|_{p}=p^{2} $, then $ X=G $ since $ O^{p'}(G)=G $. Therefore, $ G/K $ is a simple group whose Sylow  $ p $-subgroups have order $ p^{2} $.  Suppose that $ K $ is the unique $ 2 $-maximal subgroup of $ P $. Then $ P $ is cyclic by \cite[Kapitel III, Satz 8.3]{Huppert-1967}, which contradicts the assumption that $ P $ is non-abelian. Hence $ P $ has  a $ 2 $-maximal  subgroup $ U $  which is  diffrent  from $ K $.  Then $ K<UK\leq P $. Note that $ K $ is the unique  minimal normal subgroup of $ G $.  Since $ K $  satisfies the partial $ \Pi  $-property in $ G $,  we have that  $ |G:N_{G}(UK\cap G)|=|G:N_{G}(UK)| $ is a $ p $-number. If $ UK $ is a maximal subgroup of $ P $, then $ UK\unlhd G $, and hence $ UK=G $,  a contradiction. If $ UK=P $, then $ P\unlhd G $, also a contradiction.

   \vskip0.1in
   \noindent \textbf{Case 3.} $ p=2 $,  $ G/K $  is a non-$ 2 $-soluble group and $ P/K $ is isomorphic to $ Q_{8} $.
   \vskip0.1in

   In this case,  we work to obtain a contradiction. Clearly, $ |{\rm Soc}(G)/K|<|P/K|=8 $. If $ |{\rm Soc}(G)/K|=4 $, then $ P/{\rm Soc}(G) $ is a Sylow $ 2 $-subgroup of $ G/{\rm Soc}(G) $ of order $ 2 $. By \cite[Corollary 5.14]{isaacs2008finite}, $ G/{\rm Soc}(G) $ is $ 2 $-nilpotent, and thus $ G $ is $ 2 $-soluble, a contradiction. Therefore, $ |{\rm Soc}(G)/K|\leq 2 $.

   Suppose that $ |{\rm Soc}(G)/K|=2 $. Let $ T $ be  a minimal normal subgroup  of $ G $ of order $ 2 $ such that $ {\rm Soc}(G) = K \times T $.  Note that $ G/T $ is not $ 2 $-soluble. Since $ |P|=8|K| $, it follows that $ |P/T|=4|K|\geq 8 $. In view of  Lemma \ref{over}, every $ 2 $-maximal subgroup  of $ P/T $ satisfies the partial $ \Pi  $-property in $ G/T $.  By induction, we have that $ P/T\cong Q_{8} $.  This yields that $ |P| = 16 $, $ |{\rm Soc}(G)| = 4 $ and every minimal normal subgroup of $ G $ has order $ 2 $. We claim that such a group $ P $ does not exist. Clearly, $ K\times T\leq Z(G) $.  If $ L $ is the third subgroup of order $ 2 $ in $ K\times T $, then $ L $ is also a minimal normal subgroup of $ G $. With a similar argument as above, we can get that $ G/L $ is not $ 2 $-soluble  and  $ P/L\cong Q_{8} $. Suppose that $ K\leq \Phi(P) $. Then $ K\times T=\Phi(P) = Z(P) $,  and hence $ P $ can be generated by two elements.  Set  $ P = \langle x_{1}, x_{2} \rangle  $. Observe  that $ P_{1} = \langle x_{1} \rangle\Phi (P) $ and $ P_{2} = \langle x_{2} \rangle \Phi (P) $ are two maximal subgroups of $ P $ and $ P = P_{1} P_{2} = \langle x_{1} \rangle \langle x_{2} \rangle   $. Write $ y = [x_{1} , x_{2} ]  $. Note that $ P $  has nilpotence class at most 2 and $ P'\leq Z(P) $.  Consequently,  $ P'=\langle y \rangle < \Phi(P) $. Clearly, $ P'\not =K $ and $ P'\not =T $.  Hence $ P' $ is the third minimal normal subgroup of $ G $ contained in $ K\times T=Z(G) $. Then  $ G/P' $ is not $ 2 $-soluble and $ P/P'\cong Q_{8} $. This is a contradiction. Therefore,  we may suppose that  $ K \nleq  \Phi(P) $. Then there exists a maximal subgroup $ P_{3} $ of $ P $ such that $ P = K\times P_{3} $, so $ P_{3}\cong Q_{8} $. As a consequence, $ P/Z(P_{3}) $ is an elementary abelian group of order $ 8 $. Since $ |Z(P)|=4 $, we have $ K\times Z(P_{3})=Z(P)=K\times T\leq Z(G) $. Hence we can  assume that $ T=Z(P_{3}) $. Thus $ P/Z(P_{3})=P/T $ is isomorphic to $ Q_{8} $, a contradiction.

   Therefore, $ {\rm Soc}(G)=K $ is the unique minimal normal subgroup of $ G $. Let $ G/A $ be a chief factor of $ G $. Then $ K\leq A $.  Since $ O^{2'}(G)=G $,  we have that $ 2 $ divides $ |G/A| $. If $ |G/A|_{2}=2 $, then by  \cite[Corollary 5.14]{isaacs2008finite}, a Sylow $2$-subgroup of $ A/K $ must be isomorphic to an elementary abelian  group of order $ 4 $. This is impossible because  $ P/K\cong Q_{8} $.  If $ |G/A|_{2}=8 $, then $ G/A $ is not a simple group by a Theorem of Brauer and Suzuki (see \cite[Chapter XII, Theorem 7.1]{Feit}). This is a contradiction. Hence we may assume that $ |G/A|_{2}=4 $. Pick a  $ 2 $-$ G $-chief factor $ B/C $ over $ K $ and below $ A $. Then $ |B/C|=2 $. This implies that $ A $ is $ 2 $-soluble. Assume that  $ D $ is a maximal normal subgroup of $ G $ which is different from $ A $. With a similar argument, we can get that  $ |G/D|_{2}=4 $ and $ D $ is $ 2 $-soluble. It follows that $ G=AD $  is $ 2 $-soluble, a contradiction. Therefore, $ A $ is the unique maximal normal subgroup of $ G $. Then the chief factor $ G/A $ occurs in every chief series of $ G $. Clearly, $ P\cap A $ is a $ 2 $-maximal subgroup of $ P $. If $ P\cap A $ is a unique  $ 2 $-maximal subgroup of $ P $, then $ P $ is  cyclic, a contradiction. Hence there exists a  $ 2 $-maximal subgroup $ M $ of $ P $ which is different  from $ P\cap A $. By hypothesis,  $ M $ satisfies the partial $ \Pi  $-property in $ G $. Then $ |G:N_{G}(MA\cap G)|=|G:N_{G}(MA)| $ is a $ 2 $-number. If $ M(P\cap A)=P $, then $ MA=G $. This means that $ |G/A| $ has order $ 4 $, and hence $ G $ is $ 2 $-soluble,  a contradiction. If $ M(P\cap A) $ is a maximal subgroup of $ P $, then $ MA=M(P\cap A)A $ is normalized by $ P $, and so $ MA $ is a normal subgroup of $ G $. Thus $ G=MA $.  This is impossible. Our proof is now complete.
   \end{proof}

   \begin{proof}[Proof of Theorem \ref{2-mini}]
   	Combining Theorem  \ref{2-minimal} and Theorem \ref{2minimal}, we conclude that $ G $ is  $ p $-soluble with  $ p $-length  at most $ 1 $.
   \end{proof}

   \begin{proof}[Proof of Theorem \ref{2-maxi}]
   	By Theorem  \ref{2-maximal} and Theorem \ref{classification}, the conclusion follows.
   \end{proof}

\section*{Acknowledgments}

    This work is supported by the National Natural Science Foundation of China (Grant No.12071376, 11971391),  the  Fundamental Research Funds for the Central Universities (No. XDJK2020B052),  the Natural Science Foundation Project of CQ (No.cstc2021jcyj-msxmX0426) and the  Fundamental Research Funds for the Central Universities (Nos.XDJK2019C116 and XDJK2019B030).
    %The authors are very grateful to the referees who read the paper carefully and provided very helpful comments.

    %\vskip1cm

    \small

    %\bibliographystyle{unsrt}
    %\bibliographystyle{plainnat}
    %\bibliographystyle{plain}
    %\bibliography{sample}

\end{document}